\documentclass[a4paper]{amsart}

\usepackage{amssymb}

\usepackage[hidelinks]{hyperref}
\usepackage[textsize=small]{todonotes}
\usepackage{xcolor}
\usepackage{enumerate}
\usepackage{lipsum}  
\usepackage[normalem]{ulem} 
\usepackage{rotating}
\usepackage{framed}

\theoremstyle{plain}
\newtheorem{theorem}{Theorem}[section]
\newtheorem{lemma}[theorem]{Lemma}

\newtheorem{proposition}[theorem]{Proposition}

\theoremstyle{definition}

\newtheorem{question}[theorem]{Question}

\theoremstyle{remark}

\newtheorem{case}{Case}

\newcommand{\cH}{\mathcal{H}}
\newcommand{\cI}{\mathcal{I}}
\newcommand{\I}{\cI}
\newcommand{\cJ}{\mathcal{J}}
\newcommand{\J}{\cJ}

\newcommand{\cP}{\mathcal{P}}
\newcommand{\cR}{\mathcal{R}}

\newcommand{\cU}{\mathcal{U}}

\newcommand{\cW}{\mathcal{W}}

\newcommand{\fin}{\mathrm{Fin}}
\newcommand{\Fin}{\fin}

\newcommand{\Hindman}{\mathcal{H}}

\newcommand{\Ramsey}{\mathcal{R}}
\newcommand{\vdW}{\mathcal{W}}

\DeclareMathOperator{\FS}{FS}

\begin{document}

%%%%%%%%%%%%%%%%%%%%%%%%%%%%%%%%%%%%%%%%%%%%%%%%%%
%%%%% TITLE
%%%%%%%%%%%%%%%%%%%%%%%%%%%%%%%%%%%%%%%%%%%%%%%%%%

\title[Kat\v{e}tov order between some Ramsey-like  ideals]{Kat\v{e}tov order between Hindman, Ramsey, van der Waerden and summable ideals}

%%%%%%%%%%%%%%%%%%%%%%%%%%%%%%%%%%%%%%%%%%%%%%%%%%
%%%%% AUTHORS
%%%%%%%%%%%%%%%%%%%%%%%%%%%%%%%%%%%%%%%%%%%%%%%%%%

\author{Rafa\l{} Filip\'{o}w}
\address[Rafa\l{}~Filip\'{o}w]{Institute of Mathematics\\ Faculty of Mathematics, Physics and Informatics\\ University of Gda\'{n}sk\\ ul.~Wita Stwosza 57\\ 80-308 Gda\'{n}sk\\ Poland}
\email{Rafal.Filipow@ug.edu.pl}
\urladdr{\url{http://mat.ug.edu.pl/~rfilipow}}

\author{Krzysztof Kowitz}
\address[Krzysztof Kowitz]{Institute of Mathematics\\ Faculty of Mathematics\\ Physics and Informatics\\ University of Gda\'{n}sk\\ ul.~Wita  Stwosza 57\\ 80-308 Gda\'{n}sk\\ Poland}
\email{Krzysztof.Kowitz@phdstud.ug.edu.pl}

\author{Adam Kwela}
\address[Adam Kwela]{Institute of Mathematics\\ Faculty of Mathematics\\ Physics and Informatics\\ University of Gda\'{n}sk\\ ul.~Wita  Stwosza 57\\ 80-308 Gda\'{n}sk\\ Poland}
\email{Adam.Kwela@ug.edu.pl}
\urladdr{\url{https://mat.ug.edu.pl/~akwela}}

%%%%%%%%%%%%%%%%%%%%%%%%%%%%%%%%%%%%%%%%%%%%%%%%%%
%%%%% DATE
%%%%%%%%%%%%%%%%%%%%%%%%%%%%%%%%%%%%%%%%%%%%%%%%%%

\date{\today}

%%%%%%%%%%%%%%%%%%%%%%%%%%%%%%%%%%%%%%%%%%%%%%%%%%
%%%%% MSC (MATHEMATICAL SUBJECT CLASSIFICATION)
%%%%%%%%%%%%%%%%%%%%%%%%%%%%%%%%%%%%%%%%%%%%%%%%%%

\subjclass[2020]{Primary: 
03E05, % (1980-now) Other combinatorial set theory 
05D10. % (1991-now) Ramsey theory [See also 05C55] 
}

%%%%%%%%%%%%%%%%%%%%%%%%%%%%%%%%%%%%%%%%%%%%%%%%%%
%%%%% KEYWORDS
%%%%%%%%%%%%%%%%%%%%%%%%%%%%%%%%%%%%%%%%%%%%%%%%%%

\keywords{Katetov order, ideal, filter, Ramsey’s theorem for coloring graphs, Hindman’s finite sums theorem, van der Waerden’s arithmetical progressions theorem}

%%%%%%%%%%%%%%%%%%%%%%%%%%%%%%%%%%%%%%%%%%%%%%%%%%
%%%%% ABSTRACT
%%%%%%%%%%%%%%%%%%%%%%%%%%%%%%%%%%%%%%%%%%%%%%%%%%

\begin{abstract}
A family $\I$ of subsets of a set $X$ is an \emph{ideal on $X$} if it is closed under taking subsets and finite unions of its elements.
An ideal $\I$ on $X$ is below an ideal $\J$ on $Y$ in the \emph{Kat\v{e}tov order}  if there is a function $f:Y\to X$ such that $f^{-1}[A]\in\J$ for every $A\in \I$.
We show that the Hindman ideal, the Ramsey ideal and the summable ideal are pairwise incomparable in the Kat\v{e}tov order, where 
\begin{itemize}
    \item 
the \emph{Ramsey ideal} consists of those sets of pairs of natural numbers which do not contain 
a set of all pairs of any infinite set (equivalently do not contain, in a sense, any infinite complete subgraph), 
\item the \emph{Hindman ideal} consists  of those 
sets of natural numbers which do not contain any infinite set together with all finite  sums of its members (equivalently do not contain IP-sets that are considered in Ergodic Ramsey theory),
\item the \emph{summable ideal} consists of those sets of natural numbers such that the series of the reciprocals of its members is convergent.
\end{itemize}
Moreover, we show that in the Kat\v{e}tov order the above mentioned ideals are not below  the \emph{van der Waerden ideal} that  consists of those sets of natural numbers which do not contain arithmetic progressions of arbitrary finite length.
\end{abstract}

%%%%%%%%%%%%%%%%%%%%%%%%%%%%%%%%%%%%%%%%%%%%%%%%%%
%%%%% MAKE TITLE
%%%%%%%%%%%%%%%%%%%%%%%%%%%%%%%%%%%%%%%%%%%%%%%%%%

\maketitle

%%%%%%%%%%%%%%%%%%%%%%%%%%%%%%%%%%%%%%%%%%%%%%%%%%
%%%%% Table of contents
%%%%%%%%%%%%%%%%%%%%%%%%%%%%%%%%%%%%%%%%%%%%%%%%%%

%%% \setcounter{tocdepth}{1}

%%%%%%%%%%%%%%%%%%%%%%%%%%%%%%%%%%%%%%%%%%%%%%%%%%
%%%%% SECTION
%%%%%%%%%%%%%%%%%%%%%%%%%%%%%%%%%%%%%%%%%%%%%%%%%%

\section{Introduction}

The Kat\v{e}tov order is an efficient tool for studying ideals over countable sets \cite{MR2777744,MR3696069,MR2017358,MR3019575,MR3555332,MR3778962}.
Originally, the Kat\v{e}tov order (introduced by Kat\v{e}tov \cite{MR250257} in 1968) was  used  to study convergence in topological spaces, and our interest in Kat\v{e}tov order between the Hindman, Ramsey, van der Waerden and summable ideals stems from the study of sequentially compact spaces defined as, in a sense, topological counterparts of well-known combinatorial theorems: 
Ramsey’s theorem for coloring graphs, 
Hindman’s finite sums theorem  and 
van der Waerden’s arithmetical progressions theorem \cite{BergelsonZelada,MR2948679,MR1887003,MR1866012,MR1950294,MR4552506}. 
It is known \cite{naszUnified,MR4584767} that an existence of a sequentially compact space which distinguishes the above mentioned classes of spaces is reducible to a question whether particular ideals  are incomaparable in the Kat\v{e}tov order.

Beside our primary interest in the Kat\v{e}tov order described above we mention  one more strength of this order.
Using the Kat\v{e}tov order,
we  can classify non-definable  objects (like ultrafilters or maximal almost disjoint families)
using Borel ideals \cite{MR3696069}. For instance, an ultrafilter $\cU$ is a P-point if and only if the dual ideal $\cU^*$  is not Kat\v{e}tov above $\Fin^2$ (equivalently $\cU$ is a $\Fin^2$-ultrafilter as defined by  Baumgartner \cite{MR1335140}). 
It is known \cite{MR4448270} that 
an existence of an ultrafilter which distinguishes between some classes of ultrafilters is reducible to a question whether particular ideals are incomaparable in the Kat\v{e}tov order.

Below we describe the results obtain in this paper and introduce a necessary notions and notations.

We write $\omega$ to denote  the set of all natural numbers (with zero).
%and, following von Neumann, we identify each natural number with the  set of all natural numbers less than $n$. 

We write  
$[A]^2$ to denote the set of all unordered pairs of elements of $A$, 
$[A]^{<\omega}$ to denote the family of all finite subsets of $A$ and
$[A]^\omega$ to denote the family of all infinite countable subsets of $A$.

A family $\I\subseteq\cP(X)$ of subsets of a set $X$ is an \emph{ideal on $X$} if it is closed under taking subsets and finite unions of its elements, $X\notin\I$ and $\I$ contains all finite subsets of $X$.
 By $\fin(X)$ we denote the family of all finite subsets of  $X$ and we  write $\Fin$ instead of $\Fin(\omega)$. 

For an ideal $\I$ on  $X$, we write $\I^+=\{A\subseteq X: A\notin\I\}$ and call it the \emph{coideal of $\I$}, and we write $\I^*=\{X\setminus A: A\in\I\}$ and call it the \emph{filter dual to $\I$}.
It is easy to see that  
$\I\restriction A=\{A\cap B:B\in \I\}$
is an ideal on $A$ if and only if  $A\in \I^+$. 

For a set $B\subseteq\omega$, we write 
$FS(B)$ to denote the set of all finite (nonempty) sums of distinct elements of $B$
i.e.~$\FS(B)=\{\sum_{n\in F}n: F\in[B]^{<\omega}\setminus\{\emptyset\}\}$.

In this paper we are interested in the following four ideals: 
\begin{itemize}
    \item the \emph{Ramsey ideal} 
$$\Ramsey=\left\{A\subseteq[\omega]^2: \forall B\in[\omega]^\omega\, ([B]^2\not\subseteq A)\right\},$$
    \item the \emph{Hindman ideal} 
$$\Hindman=\left\{A\subseteq\omega: \forall B\in[\omega]^\omega\, \FS(B)\not\subseteq A\right\},$$
    \item the \emph{van der Waerden ideal} 
\begin{equation*}
 \begin{split}
\vdW= \{A\subseteq\omega:  &  \text{ $A$ does not contain arithmetic progressions} 
\\& 
\text{ of arbitrary finite length}\},    
 \end{split}   
\end{equation*}
    \item the \emph{summable ideal}
$$\mathcal{I}_{1/n}=\left\{A\subseteq\omega: \sum_{n\in A}\frac{1}{n+1}<\infty\right\}.$$
\end{itemize}

The Ramsey ideal was introduced by Meza-Alc\'{a}ntara and Hru\v{s}\'{a}k  \cite{MR3019575} (the authors noted that if we identify a set $A\subseteq [\omega]^2$ with a graph $G_A=(\omega,A)$, the ideal $\Ramsey$ can be seen as an ideal consisting of graphs without infinite complete subgraphs).
Both the Hindman and van der Waerden ideals were introduced by Fla\v{s}kova \cite[p.~109]{MR2471564}.
The  summable ideal is a particular instance of the so-called \emph{summable ideals} which 
seem to be ``ancient'' compare to previously mentioned ideals as they were introduced in 1972 by Mathias  \cite[Example~3, p.206]{MR0363911}.

We say that an ideal $\I$ on $X$ is below an ideal $\J$ on $Y$ in the \emph{Kat\v{e}tov order} \cite{MR250257}  
 if there is a function $f:Y\to X$ such that $f^{-1}[A]\in\J$ for every $A\in \I$ (equivalently, $f[B]\notin\I$ for all $B\notin\J$).
Note that the Kat\v{e}tov order  has been extensively examined (even in its own right) for many years so far \cite{
MR3034318,
MR1335140,
MR3600759,
MR4378082,
MR4247792,
MR4036731,
MR3513296,
MR3692233,
MR2777744,
MR3696069,
MR2017358,
MR3019575,
alcantara-phd-thesis,
MR3555332,
MR3550610,
MR3778962,
MR4312995}.

The aim of this  paper is to prove  the following 
\begin{theorem}\ 
\label{thm:THEOREM}
\begin{enumerate}
    \item 
The ideals  $\Ramsey$, $\Hindman$ and $\I_{1/n}$ are pairwise incomparable in the Kat\v{e}tov order.
\item The ideals  $\Ramsey$, $\Hindman$ and $\I_{1/n}$ are not below the ideal $\vdW$ in the Kat\v{e}tov order. 
\end{enumerate}    
\end{theorem}

As far as we are concerned, the remaining three questions about these ideals are still open.

\begin{question}
\label{q1}
\label{q2}
\label{q3}
Is the ideal $\vdW$ below the ideal $\Ramsey$ ($\Hindman$, $\I_{1/n}$, resp.) in the Kat\v{e}tov order?
\end{question}

Note that in the case of the summable ideal, Question~\ref{q3} is a weakening of the famous Erd\H{o}s-Tur\'{a}n conjecture which says that $\vdW\subseteq\mathcal{I}_{1/n}$.

%%%%%%%%%%%%%%%%%%%%%%%%%%%%%%%%%%%%%%%%%%%%%%%%%%
%%%%% SECTION
%%%%%%%%%%%%%%%%%%%%%%%%%%%%%%%%%%%%%%%%%%%%%%%%%%

\section{Preliminaries}

An ideal $\I$ on $X$  is 
\emph{tall} \cite[Definition~0.6]{MR491197} if for every infinite set $A\subseteq X$ there exists an infinite set  $B\subseteq A$ such that $B\in \I$.
It is not difficult to see that $\I$ is not tall $\iff$  $\I\leq_K\J$ for every ideal $\J$ $\iff$
$\I\leq_K\Fin$ $\iff$ $\I\restriction A=\Fin(A)$ for some $A\in\I^+$.
It is easy to show the following   
\begin{proposition}
\label{prop:tall}
The ideals $\Hindman$, $\Ramsey$, $\vdW$ and $\I_{1/n}$ are tall.
\end{proposition}

Ideals  $\I$ and $\J$ on  $X$ and $Y$, respectively are \emph{isomorphic} 
(in short: $\I\approx \J$) if 
 there exists  a bijection $\phi:X\to Y$ 
 such that $ A\in \I \iff \phi[A] \in \J$ for each $A\subseteq X$. An ideal  $\I$ is \emph{homogeneous} \cite[Definition~1.3]{MR3594409} if 
 the ideals $\I$ and $\I\restriction A$ are isomorphic for every $A\in \I^+$.

\begin{proposition}[{\cite[Examples 2.5 and 2.6]{MR3594409}}]
\label{prop:homogeneous}
The ideals $\Hindman$, $\Ramsey$ and $\vdW$ are homogeneous.
\end{proposition}

By identifying subsets of $X$  with their characteristic functions,
we equip $\cP(X)$ with the topology of the space $2^X$ (the product topology of countably many copies of the discrete topological space $\{0,1\}$) and therefore
we can assign topological notions to ideals on $X$.
In particular, an ideal $\I$ is \emph{Borel} (\emph{$F_\sigma$}, resp.) if $\I$ is a Borel  ($F_\sigma$, resp.) subset of $2^X$.

If $A\subseteq\omega$ and $n\in \omega$, we write
$A+n = \{a+n:a\in A\}$ and $A-n=\{a-n:a\in A,a\geq n\}$. 

A set $D \subseteq \omega$ is \emph{sparse} \cite[p.~1598]{MR1887003} if for each $x\in \FS(D)$ there exists the unique set $\alpha \subseteq D$ such that $x = \sum_{n\in \alpha} n$. This unique set will be denoted by $\alpha_D(x)$.
For instance, the set $E = \{2^n:n\in\omega\}$ is sparse, and in the sequel, we write $\alpha(x)$ instead of $\alpha_E(x)$.

A set $D \subseteq \omega$ is \emph{very sparse} \cite[p.~894]{MR4356195} if it is sparse and 
$$\forall x,y\in \FS(D)\, \left(\alpha_D(x)\cap \alpha_D(y) \neq\emptyset \implies x+y\notin \FS(D)\right).$$ In the sequel, we will use the following 
\begin{lemma}[{\cite[Lemma 2.2]{MR4356195}}]
\label{lem:VERY-SPARSE}
For every infinite set $D \subseteq \omega$ there is an infinite set $D'\subseteq D$ which is very sparse.
\end{lemma}

%%%%%%%%%%%%%%%%%%%%%%%%%%%%%%%%%%%%%%%%%%%%%%%%%%
%%%%% SECTION
%%%%%%%%%%%%%%%%%%%%%%%%%%%%%%%%%%%%%%%%%%%%%%%%%%

\section{Summable and van der Waerden ideals are not above Hindman and Ramsey ideals}

To show that $F_\sigma$ ideals are not above $\Hindman$ nor $\Ramsey$ in the Kat\v{e}tov order one can use the following ideal on $\omega^2$ introduced by Kat\v{e}tov \cite[Definition 5.1]{MR0355956}:
$$\Fin^2=\left\{C\subseteq\omega^2: 
\{n\in \omega: \{k\in\omega:(n,k)\in C\}\notin \Fin\}\in \Fin\right\}.$$

The following lemma and proposition can be found in \cite{naszUnified}, but we decided to include proofs here  for the sake of  completeness.

\begin{lemma}[{\cite[Proposition~7.2]{naszUnified}}]\
\label{lemat}
\begin{itemize}
\item[(a)] $\Fin^2\leq_K \cH$.
\item[(b)] $\Fin^2\leq_K \cR$.
\end{itemize}
\end{lemma}

\begin{proof}
(a): Let $A_k=\{2^k(2n+1): n\in\omega\}$ for each $k\in\omega$. Let $f:\omega\to\omega^2$ be any injective function such that $f[A_k]\subseteq\{k\}\times\omega$ for all $k\in\omega$. In  \cite[item (2) in the proof of Proposition~1.1]{MR4356195}, the authors showed that $A_k\in \Hindman$ for every $k\in\omega$ (so $f^{-1}[\{k\}\times\omega]\in\Hindman$ for all $k\in\omega$), whereas in 
\cite[item (1) in the proof of Proposition~1.1]{MR4356195} it is shown that 
for every $B\notin\Hindman$  there is $k\in\omega$ such that $B\cap A_k$ is infinite (so $f^{-1}[C]\in\Hindman$ whenever $C\subseteq\omega^2$ is such that $C\cap(\{k\}\times\omega)$ is finite for all $k\in\omega$). 
Thus, the function $f$ witnesses the fact that $\Fin^2\leq_K \cH$. 

(b): Let $A_n = \{\{k,i\}: i>k\geq n\}$ for every $n\in\omega$. 
Then $A_n\notin\Ramsey$, $A_0=[\omega]^2$, $\bigcap_{n\in\omega}A_n=\emptyset$ and $A_n\setminus A_{n+1} = \{\{n,i\}: i> n\}\in \Ramsey$. Let $f:[\omega]^2\to\omega^2$ be any injective function such that $f[A_n\setminus A_{n+1}]\subseteq\{n\}\times\omega$. Then $f^{-1}[\{n\}\times\omega]\in\Ramsey$ for all $n\in\omega$. 
Suppose, for sake of contradiction, that there is $C\subseteq\omega^2$ such that $C\cap(\{k\}\times\omega)$ is finite for all $k\in\omega$ (so $C\in\Fin^2$), but $B=f^{-1}[C]\notin\Ramsey$. Then $B\subseteq^* A_n$ for every $n\in\omega$.
Let $H=\{h_n:n\in\omega\}$ be an infinite set such that $[H]^2\subseteq B$ and $h_n<h_{n+1}$ for every $n\in \omega$.
Since $[H]^2\subseteq^*A_{h_1}$, there is a finite set $F$ such that 
$[H]^2\setminus F\subseteq A_{h_1}$.
Since $F$ is finite, there is $k>0$ such that $\{h_0,h_n\}\notin F$ for every $n\geq k$.
Then  $\{\{h_0,h_n\}:n\geq k\}\subseteq [H]^2\setminus F$ and $\{\{h_0,h_n\}:n\geq k\}\cap A_{h_1}=\emptyset$, a contradiction.
\end{proof}

\begin{proposition}[{\cite[Theorem~7.7]{naszUnified}}]\
\begin{itemize}
\item[(a)] $\mathcal{H}\not\leq_K\mathcal{W}$.
\item[(b)] $\mathcal{R}\not\leq_K\mathcal{W}$.
\item[(c)] $\mathcal{H}\not\leq_K\mathcal{I}_{1/n}$.
\item[(d)] $\mathcal{R}\not\leq_K\mathcal{I}_{1/n}$.
\end{itemize}
\end{proposition}

\begin{proof}
(a): Suppose otherwise: $\cH\leq_K\cW$. 
Using Lemma \ref{lemat} we get that $\Fin^2\leq_K \cH$, so $\Fin^2\leq_K\cW$. However, since $\cW$ is $F_\sigma$ (see \cite[Example~4.12]{MR4572258}), $\Fin^2\not\leq_K\cW$ (by \cite[Theorems 7.5 and 9.1]{MR2520152} and \cite[Example 4.1]{MR3034318}). A contradiction.

The proofs of items (b), (c) and (d) are similar to the proof of item (a), since $\Fin^2\leq_K \cR$ (by Lemma \ref{lemat}) and $\I_{1/n}$ is $F_\sigma$ (see \cite[Example~1.5]{MR1124539}).
\end{proof}

%%%%%%%%%%%%%%%%%%%%%%%%%%%%%%%%%%%%%%%%%%%%%%%%%%
%%%%% SECTION
%%%%%%%%%%%%%%%%%%%%%%%%%%%%%%%%%%%%%%%%%%%%%%%%%%

\section{Summable ideal is not below van der Waerden ideal}

\begin{proposition}
$\mathcal{I}_{1/n}\not\leq_K\mathcal{W}$.
\end{proposition}

\begin{proof}
Suppose for sake of contradiction that 
there is a function  $\phi:\omega\to\omega$
such that $\phi^{-1}[B]\in \vdW$ for every $B\in \I_{1/n}$. 
We construct a sequence $(F_n:n\in\omega)$ of finite subsets of $\omega$ such that for every $n\in\omega$ we have
\begin{enumerate}
    \item $F_n$ is an arithmetic progression of length $n$,
    \item $\phi(x)\geq n2^n$ for every $x\in F_n$.
\end{enumerate}
Suppose that $F_i$ are constructed for $i<n$.
Since $B=\{i\in\omega: i<n2^n\}$ is finite,  $A=\phi^{-1}[B]\in \vdW$.
Then $\omega\setminus A\notin \vdW$, so there is an arithmetic progression $F_n\subseteq \omega\setminus A$ of length $n$. This finishes the construction of $F_n$.

Let $A=\bigcup\{F_n:n\in\omega\}$. Then $A\notin\vdW$, but 
$$
\sum_{y\in \phi[A]} \frac{1}{y+1} 
\leq 
\sum_{n\in\omega} \left(\sum_{x\in F_n} \frac{1}{\phi(x)+1} \right) 
\leq 
\sum_{n\in\omega} \left(\sum_{x\in F_n} \frac{1}{n2^n+1} \right) 
= 
\sum_{n\in\omega} \frac{n}{n2^n+1} <\infty, 
$$
so $\phi[A]\in \I_{1/n}$, a contradiction.
\end{proof}

%%%%%%%%%%%%%%%%%%%%%%%%%%%%%%%%%%%%%%%%%%%%%%%%%%
%%%%% SECTION
%%%%%%%%%%%%%%%%%%%%%%%%%%%%%%%%%%%%%%%%%%%%%%%%%%

\section{Summable ideal is not below Hindman ideal}

\begin{theorem}
$\I_{1/n} \not\leq_K \Hindman$.
\end{theorem}

\begin{proof}
This is proved in \cite[Theorem~3.2]{MR4356195}, but below we provide a simpler proof.

Let $\phi:\omega\to\omega$ be an arbitrary function.
We will show that $\phi$ is not a witness for  $\I_{1/n}\leq_K \Hindman$
 i.e.~we will find an infinite set $D\subseteq\omega$ such that $\phi[\FS(D)]\in\I_{1/n}$.
 
Using Canonical Hindman Theorem (\cite[Theorem 2.1]{MR424571}, see also \cite[Theorem~5 at p.~133]{MR1044995}),
there is an infinite set $C=\{c_n:n\in\omega\}\subseteq\omega$ such that 
$\max\alpha(c_n)<\min\alpha(c_{n+1})$ for every $n\in\omega$ and 
one of  the following five  cases holds:
\begin{enumerate}
    \item $\forall  x,y\in \FS(C) (\phi(x)=\phi(y))$, 
    \item $\forall  x,y\in \FS(C) (\phi(x)=\phi(y) \iff \min \alpha(x) =\min \alpha(y))$, 
    \item $\forall  x,y\in \FS(C) (\phi(x)=\phi(y) \iff \max \alpha(x) =\max \alpha(y))$, 
    \item $\forall  x,y\in \FS(C) (\phi(x)=\phi(y) \iff (\min \alpha(x) =\min \alpha(y) \text{ and } \max \alpha(x) =\max \alpha(y)))$, 
    \item $\forall  x,y\in \FS(C) (\phi(x)=\phi(y) \iff x=y)$.
\end{enumerate}

\emph{Case 1.}
We take $D=C$ and see that the set $\phi[\FS(D)]$ has only one element, so it belongs to $\I_{1/n}$.

\emph{Case 2.}
We construct a strictly increasing  sequence $\{k_n:n\in\omega\}$ such that 
$\phi(c_{k_n})> 2^{n}$ for every $n\in\omega$.

Suppose that $k_i$ are constructed for $i<n$.
Since $\max\alpha(c_k)<\min\alpha(c_{k+1})$ for every $k\in\omega$, 
$\min\alpha(c_k)\neq\min\alpha(c_l)$ for distinct $k,l\in\omega$.
Consequently, $\phi\restriction C$ is one-to-one, so we can find $k_n>k_{n-1}$ such that $\phi(c_{k_n})>2^{n}$.
That finishes the inductive construction of $k_n$.

Let $D=\{c_{k_n}:n\in\omega\}$. If we show that  $\phi[\FS(D)]\in\I_{1/n}$, the proof of this case will be finished. Using the properties of $c_{k_n}$'s we can see that
$\phi[c_{k_n}+\FS(\{c_{k_i}:i>n\})] = \{\phi(c_{k_n})\},$
for every $n\in\omega$, so 
\begin{equation*}
    \begin{split}
\sum_{y\in \phi[\FS(D)]} \frac{1}{y+1} 
&=
\sum_{n\in\omega}\left(\sum_{y\in \{\phi(c_{k_n})\}\cup\phi[c_{k_n}+\FS(\{c_{k_i}:i>n\})]}\frac{1}{y+1}\right)
\\&=
\sum_{n\in\omega}\frac{1}{\phi(c_{k_n})+1}
\leq 
\sum_{n\in\omega} \frac{1}{2^{n}+1}
< \infty.
\end{split} 
\end{equation*}

\emph{Case 3.}
We construct a strictly increasing  sequence $\{k_n:n\in\omega\}$ such that 
$\phi(c_{k_n})> 2^{n}$ for every $n\in\omega$.

Suppose that $k_i$ are constructed for $i<n$.
Since $\max\alpha(c_k)<\min\alpha(c_{k+1})$ for every $k\in\omega$, 
$\max\alpha(c_k)\neq\max\alpha(c_l)$ for distinct $k,l\in\omega$.
Consequently, $\phi\restriction C$ is one-to-one, so we can find $k_n>k_{n-1}$ such that $\phi(c_{k_n})>2^{n}$.
That finishes the inductive construction of $k_n$.

Let $D=\{c_{k_n}:n\in\omega\}$. If we show that  $\phi[\FS(D)]\in\I_{1/n}$, the proof of this case will be finished. Using the properties of $c_{k_n}$'s we can see that
$\phi[c_{k_n}+\FS(\{c_{k_i}:i<n\})] = \{\phi(c_{k_n})\},$
for every $n\in\omega$, so 
\begin{equation*}
    \begin{split}
\sum_{y\in \phi[\FS(D)]} \frac{1}{y+1} 
&=
\sum_{n\in\omega}\left(\sum_{y\in \{\phi(c_{k_n})\}\cup\phi[c_{k_n}+\FS(\{c_{k_i}:i<n\})]}\frac{1}{y+1}\right)
\\&=
\sum_{n\in\omega}\frac{1}{\phi(c_{k_n})+1}
\leq 
\sum_{n\in\omega}\frac{1}{2^{n}+1}
< \infty.
\end{split}
\end{equation*}

\emph{Case 4.}
We construct a strictly increasing  sequence $\{k_n:n\in\omega\}$ such that 
$$ \forall n\in\omega\,\forall i<n\,  \left( \phi(c_{k_n})>n2^{n} \land \phi(c_{k_n}+c_{k_i})>n2^{n}\right).$$

Suppose that $k_i$ are constructed for $i<n$.
Since $\max\alpha(c_k)<\min\alpha(c_{k+1})$ for every $k\in\omega$, 
we obtain that
$\min\alpha(c_k+c_{k_i})\neq \min\alpha(c_k+c_{k_j})$ and 
$\min\alpha(c_k)\neq \min\alpha(c_k+c_{k_i})$
for every $k>k_{n-1}$ and $i<j\leq n-1$.
Consequently, 
the function $\phi\restriction (\{c_{k}+c_{k_i}:k>k_{n-1}, i<n\}\cup\{c_k:k>k_{n-1}\})$
is one-to-one, so using pigeonhole principle 
we can find $k_n>k_{n-1}$ such that $\phi(c_{k_n})>n2^{n}$
and 
$\phi(c_{k_n}+c_{k_i})> n2^{n}$
for every $i<n$.
That finishes the inductive construction of $k_n$.

Let $D=\{c_{k_n}:n\in\omega\}$. If we show that  $\phi[\FS(D)]\in\I_{1/n}$, the proof of this case will be finished. Using the properties of $c_{k_n}$'s we can see that
$\phi[c_{k_m}+\FS(\{c_{k_i}:m<i<n\})+c_{k_n}] = \{\phi(c_{k_m}+c_{k_n})\}$
for every $m<n$, $m,n\in \omega$, so 
\begin{equation*}
    \begin{split}
\sum_{y\in \phi[\FS(D)]} \frac{1}{y+1} 
&=
\sum_{n\in\omega}\frac{1}{\phi(c_{k_n})+1} 
\\&+
\sum_{n\in\omega}\sum_{m<n}\left(\sum_{y\in\{\phi(c_{k_m}+c_{k_n})\}\cup \phi[c_{k_m}+\FS(\{c_{k_i}:m<i<n\})+c_{k_n}]}\frac{1}{y+1}\right)
\\&=
\sum_{n\in\omega} \frac{1}{\phi(c_{k_n})+1} 
+
\sum_{n\in\omega}\sum_{m<n}\frac{1}{\phi(c_{k_m}+c_{k_n})+1}
\\&\leq 
\sum_{n\in\omega} \frac{1}{n2^{n}+1} 
+ 
\sum_{n\in\omega}\sum_{m<n} \frac{1}{n2^{n}+1}
<\infty.
    \end{split}
\end{equation*}

\emph{Case 5.}
We construct inductively a strictly increasing  sequence $\{k_n:n\in\omega\}$ such that 
$$
\forall n\in\omega\, \forall x\in \FS(\{c_{k_i}:i<n\})\, 
\left(\phi(c_{k_n})> 2^{2n}
\land 
\phi(c_{k_n}+x)> 2^{2n}\right).$$

Suppose that $k_i$ are constructed for $i<n$.
Let $m\in\omega$ be such that 
$m>2^{2n}$ and $m>\phi(x)$ for every $x\in \FS(\{c_{k_i}:i<n\})$.
Since $\phi\restriction \FS(C)$ is one-to-one, the set $F=\phi^{-1}[\{0,1,\dots, m\}]$ is finite.
Let $k_n\in\omega$ be such that $c_{k_n}>\max F$.
Since $c_{k_n}>\max F$,  we obtain that $c_{k_n} \notin F$ and consequently $\phi(c_{k_n})>m> 2^{2n}$.
Similarly, for every $x\in \FS(\{c_{k_i}:i<n\})$ we have $c_{k_n}+x>c_{k_n}>\max F$, so $\phi(c_{k_n}+x)> m>2^{2n}$.
That finishes the inductive construction of $k_n$.

Let $D=\{c_{k_n}:n\in\omega\}$. If we show that  $\phi[\FS(D)]\in\I_{1/n}$, the proof of this case will be finished. Using the properties of $c_{k_n}$'s we can see that:
\begin{equation*}
    \begin{split}
\sum_{y\in \phi[\FS(D)]} \frac{1}{y+1} 
&=
\sum_{n\in\omega}\left(\frac{1}{\phi(c_{k_n})+1}+\sum_{x\in \FS(\{c_{k_i}:i<n\})}\frac{1}{\phi(c_{k_n}+x)+1}\right)
\\&\leq 
\sum_{n\in\omega}\left(\frac{1}{2^{2n}+1}+\sum_{x\in \FS(\{c_{k_i}:i<n\})}\frac{1}{2^{2n} +1}\right)
\\&\leq 
\sum_{n\in\omega}\left(\frac{1}{2^{2n}+1}+(2^n-1)\cdot \frac{1}{2^{2n} +1}\right)
<\infty.\qedhere 
\end{split}
\end{equation*}   
\end{proof}

\section{Summable ideal  is not below Ramsey ideal}

\begin{theorem}
$\I_{1/n} \not\leq_K \Ramsey$.
\end{theorem}

\begin{proof}
Let $\phi:[\omega]^2\to\omega$ be an arbitrary function.
We will show that $\phi$ is not a witness for  $\I_{1/n}\leq_K \Ramsey$ i.e.~we will find an infinite set $H\subseteq\omega$ such that $\phi[[H]^2]\in\I_{1/n}$.
Using Canonical Ramsey Theorem (\cite[Theorem II]{MR37886}, see also \cite[Theorem~2 at p.~129]{MR1044995}), there is an infinite set $T\subseteq\omega$ such that 
one of  the following four cases holds:
\begin{enumerate}
    \item $\forall  x,y\in [T]^2 (\phi(x)=\phi(y))$, 
    \item $\forall  x,y\in [T]^2 (\phi(x)=\phi(y) \iff \min x=\min y)$, 
    \item $\forall  x,y\in [T]^2 (\phi(x)=\phi(y) \iff \max x=\max y)$,
    \item $\forall  x,y\in [T]^2 (\phi(x)=\phi(y) \iff x=y)$.
\end{enumerate}

%\begin{enumerate}
%    \item $\forall  \{t_1,t_2\},\{t_3,t_4\}\in [T]^2 (\phi(\{t_1,t_2\})=\phi(\{t_3,t_4\}))$, 
%    \item $\forall  \{t_1,t_2\},\{t_3,t_4\}\in [T]^2 (\phi(\{t_1,t_2\})=\phi(\{t_3,t_4\}) \iff \{t_1,t_2\}=\{t_3,t_4\})$, 
%    \item $\forall  \{t_1,t_2\},\{t_3,t_4\}\in [T]^2 (\phi(\{t_1,t_2\})=\phi(\{t_3,t_4\}) \iff \min\{t_1,t_2\}=\min\{t_3,t_4\})$, 
%    \item $\forall  \{t_1,t_2\},\{t_3,t_4\}\in [T]^2 (\phi(\{t_1,t_2\})=\phi(\{t_3,t_4\}) \iff \max\{t_1,t_2\}=\max\{t_3,t_4\})$.
%\end{enumerate}

\emph{Case 1.}
We take $H=T$ and see that the set $\phi[[H]^2]$ has only one element, so it belongs to $\I_{1/n}$.

\emph{Case 2.}
In this case, for every $t\in T$ the restriction
$\phi\restriction \{\{t,s\}: s\in T, s>t\}$
is constant with  distinct values  for distinct $t$.
Thus, for every $t\in T$ there is $k_t$ such that 
 $\{k_t\}=\phi[\{\{t,s\}:s\in T, s>t\}]$.
 
Since $k_{t_n}$ are pairwise distinct, we can find a one-to-one sequence $\{t_n:n\in\omega\}\subseteq T$
such that 
 $k_{t_n}>2^n$
for every $n\in\omega$.

Now, we take $H=\{t_n:n\in\omega\}$ and notice that 
$$
\sum_{k\in \phi[[H]^2]}\frac{1}{k+1} 
= 
\sum_{n=0}^\infty\left(\sum_{k\in \phi[\{\{t_n,t_i\}:i>n\}]}\frac{1}{k+1}\right) 
= 
\sum_{n=0}^\infty \frac{1}{k_{t_n}+1}
\leq 
\sum_{n=0}^\infty \frac{1}{ 2^n}
<\infty,
$$
so $\phi[[H]^2]\in \I_{1/n}$.

\emph{Case 3.}
In this case, for every $t\in T$ the restriction
$\phi\restriction \{\{s,t\}: s\in T, s<t\}$
is constant with  distinct values  for distinct $t$.
Thus, for every $t\in T$ there is $k_t$ such that 
 $\{k_t\}=\phi[\{\{t,s\}:s\in T, s<t\}]$.

Since $k_{t_n}$ are pairwise distinct, we can find a one-to-one sequence $\{t_n:n\in\omega\}\subseteq T$
such that 
 $k_{t_n}>2^n$
for every $n\in\omega$.

Now, we take $H=\{t_n:n\in\omega\}$ and notice that 
$$
\sum_{k\in \phi[[H]^2]}\frac{1}{k+1} 
= 
\sum_{n=0}^\infty\left(\sum_{k\in \phi[\{\{t_i,t_n\}:i<n\}]}\frac{1}{k+1}\right) 
= 
\sum_{n=0}^\infty \frac{1}{k_{t_n}+1}
\leq 
\sum_{n=0}^\infty \frac{1}{ 2^n}
<\infty,
$$
so $\phi[[H]^2]\in \I_{1/n}$.

\emph{Case 4.}
We construct inductively a one-to-one sequence $\{t_n:n\in\omega\}\subseteq T$
such that 
$\phi(\{t_i,t_n\})>n\cdot 2^n$
for every $n\in\omega$ and every $i<n$.

Suppose that $t_i$ are constructed for $i<n$.
Since there are only finitely many numbers below $n\cdot 2^n$ and the function $\phi$ is one-to-one on $[T]^2$ there is $t_n\in T\setminus \{t_i: i<n\}$ such that 
$\phi(\{t,t_n\})>n\cdot 2^n$ for every $t\in T$. That finishes the inductive construction of $t_n$.

Now, we take $H=\{t_n:n\in\omega\}$ and notice that 
$$
\sum_{k\in \phi[[H]^2]}\frac{1}{k+1} 
= 
\sum_{n=0}^\infty \sum_{i<n}\frac{1}{\phi(\{t_i,t_n\})+1}
\leq 
\sum_{n=0}^\infty \sum_{i<n}\frac{1}{n\cdot 2^n+1}
\leq
\sum_{n=0}^\infty \frac{1}{2^n}<\infty,
$$
so $\phi[[H]^2]\in \I_{1/n}$.
\end{proof}

\section{Hindman ideal is not below Ramsey ideal}

\begin{lemma}\  
\label{lem:VERY-SPARSE2}
If $D$ is very sparse, then 
$\{x\in \FS(D): \alpha_D(x)\cap\alpha_D(y)\neq\emptyset\}\in \Hindman$
for every $y\in \FS(D)$.
\end{lemma}

\begin{proof}
Let $(d_n)_{n\in \omega}$ be the increasing enumeration of all elements of $D$ and  $\alpha_D(y) = \{k_0,\dots,k_n\}$.
Since 
$$
\{x\in \FS(D): \alpha_D(x)\cap\alpha_D(y)\neq\emptyset\}
=
\bigcup_{i\leq n}\{x\in \FS(D): k_i\in \alpha_D(x)\},
$$
we only need to show that 
$\{x\in \FS(D): k_i\in \alpha_D(x)\}\in\Hindman$ for every $i\leq n$.

If $y,z\in \{x\in \FS(D): k_i\in \alpha_D(x)\}$, then 
$k_i\in \alpha_D(y)\cap \alpha_D(z)\neq\emptyset$, so $y+z\notin \FS(D)$ (since $D$ is very sparse). Thus, there is no infinite (even two-element) set $C$ such that $\FS(C)\subseteq \{x\in \FS(D): k_i\in \alpha_D(x)\}$.
\end{proof}

\begin{theorem}
\label{thm:Hindman-not-below-Ramsey}
    $\Hindman \not\leq_K \Ramsey$.
\end{theorem}

\begin{proof}
Let $D\subseteq\omega$ be a very sparse set (which exists by Lemma \ref{lem:VERY-SPARSE}).
Since the ideal $\Hindman$ is homogeneous (see Proposition \ref{prop:homogeneous}), it suffices to show that $\Hindman \restriction \FS(D)\not\leq_K\Ramsey$.

Assume to the contrary that 
there exists $f:[\omega]^2\to \FS(D)$ which witnesses 
$\Hindman \restriction \FS(D) \leq_K\Ramsey$.

We will recursively define infinite sets  $B_n \subseteq \omega$  and pairwise distinct elements $b_n\in \omega$ such that for all $n\in\omega$ the following conditions are satisfied:
\begin{enumerate}[(a)]
    \item $b_n\in B_n$, $b_{n+1}>b_n$,
    
    \item $B_{n+1}\subseteq B_{n}$, $B_{0}=\omega$,\label{thm:Ramsey-not-Hindman:item-1}

    \item
    for each $y \in f\left[[\{b_i: i<n\}]^2\right]$ we have\label{thm:Ramsey-not-Hindman:item-2}
    $$f[[B_n]^2]\cap \{x\in \FS(D):\ \alpha_D(x)\cap \alpha_D(y)\neq\emptyset\}=\emptyset,$$ 

    \item
    for each  $y \in f\left[[\{b_i: i< n\}]^2\right]$ and $i< n$ we have\label{thm:Ramsey-not-Hindman:item-3}
    $$f\left[\{\{b_i,b\}:b\in B_n\}\right] -y\in \Hindman.$$
\end{enumerate}

Let $b_0=0$ and $B_0=\omega$. Then $b_0$ and $B_0$ are as required.
Assume that $b_i$ and $B_i$ have been constructed for $i<n$ and satisfy items (\ref{thm:Ramsey-not-Hindman:item-1})--(\ref{thm:Ramsey-not-Hindman:item-3}).

Since 
$\{x\in \FS(D): \alpha_D(x)\cap \alpha_D(y)\neq\emptyset\}\in \Hindman$  for every $y\in f\left[[\{b_i: i<n\}]^2\right]\subseteq\FS(D)$ (by Lemma~\ref{lem:VERY-SPARSE2}),
$[B_{n-1}]^2\in\Ramsey^+$ and we assumed that $f$ witnesses $\Hindman\restriction \FS(D)\leq_K\Ramsey$, 
there exists an infinite set  $B \subseteq \omega$ such that 
$$[B]^2 \subseteq [B_{n-1}]^2\setminus \bigcup_{y\in f\left[[\{b_i: i< n\}]^2\right]} f^{-1}\left[\{x\in \FS(D):\ \alpha_D(x)\cap \alpha_D(y)\neq\emptyset\}\right].$$

Observe that for each infinite set $E\subseteq \omega$ and $b,y\in \omega$ there exists an infinite set $C\subseteq E$ such that $f[\{ \{ b,c\}: c\in C\}] -y\in \Hindman.$
Indeed, let $g:E\setminus \{b\}\to \omega$ be given by
$g(x) = f(\{b,x\}) - y $.
Since $\Hindman$ is a tall ideal (Proposition \ref{prop:tall}), $\Hindman\not\leq_K \fin(E\setminus\{b\})$. Thus, there is $C\notin \Fin(E\setminus \{b\})$ such that $C\subseteq E\setminus \{b\}$ and $g[C]= f[\{\{ b,c\}: c\in C\}] - y\in \Hindman$. 

Now, using recursively the above observation we can find an infinite set $C\subseteq B$ such that $f[\{ \{ b_i,c\}: c\in C\}]-y\in \Hindman$
for every $i<n$ and $y \in f\left[[\{b_i: i< n\}]^2\right]$.

We put  $B_n=C$ and pick any $b_n\in B_n$ with $b_n>b_{n-1}$.

The construction of the sequences $(B_n)_{n\in \omega}$ and $(b_n)_{n\in \omega}$ is finished.

Let  $B=\{b_n:\ n\in\omega\}$. 
Since $B$ is infinite, $[B]^2\in\Ramsey^+$. 
Since we assumed that $f$ witnesses $\Hindman\restriction \FS(D)\leq_K\Ramsey$, 
 $f[[B]^2]\in\Hindman^+\restriction\FS(D)$, and consequently  there exists an infinite set $C\subseteq\omega$ such  that  
 $\FS(C)\subseteq f[[B]^2]$. 

Pick any $c\in C$ and let  $j,n\in\omega$ be such that $c=f(\{b_{j},b_{n}\})$ and $j<n$.

Since $X = \left[\{b_i: i\leq n\}\right]^2$ is finite, 
$f[X] - c \in \Hindman $.

Let 
$Y = \{\{b_i,b_k\}: i\leq n <  k\}\}$. 
Since $\{b_k:k>n\}\subseteq B_{n+1}$ and $B_{n+1}$ satisfies  item~(\ref{thm:Ramsey-not-Hindman:item-3}) applied to $y=c$, we have 
$f\left[Y\right] - c \in \Hindman$.

Let 
$Z = [\{b_i: i>n\}]^2$. 
We claim that 
$\FS(C\setminus\{c\}) \cap \left(f\left[ Z \right] - c\right) = \emptyset$. 
Suppose to the contrary that there exists  $a\in \FS(C\setminus\{c\})\cap \left(f\left[ Z \right] - c\right) $.
Then $a+c \in \FS(C)\cap f[Z]\subseteq\FS(D)\cap f[[B_{n+1}]^2]$, so 
by item~(\ref{thm:Ramsey-not-Hindman:item-2}) applied to $y=c$,
$\alpha_D(c)\cap \alpha_D(a+c)=\emptyset$.
On the other hand, $a,c\in \FS(D)$, $D$ is very sparse and $a+c\in \FS(D)$, so $\alpha_D(a)\cap \alpha_D(c)=\emptyset$. Consequently, $\alpha_D(a+c) = \alpha_D(a)\cup \alpha_D(c)$, so 
$\alpha_D(c)\cap \alpha_D(a+c) = \alpha_D(c) \neq\emptyset$, a contradiction.

Since $[B]^2 = X\cup Y\cup Z$
and 
$\FS(C\setminus\{c\})\subseteq \FS(C) -c \subseteq f[[B]^2] -c$, we have
\begin{equation*}
    \begin{split}
\FS(C\setminus\{c\})
&\subseteq 
(f[X]-c) \cup (f[Y]-c) \cup ((f[Z]-c)\cap \FS(C\setminus\{c\}))
\\&=
(f[X]-c) \cup (f[Y]-c) \cup \emptyset \in\Hindman,
    \end{split}
\end{equation*}
a contradiction.
\end{proof}

\section{Ramsey  ideal is not below Hindman ideal}

\begin{theorem}
\label{thm:Ramsey-not-below-Hindman}
$\Ramsey \not\leq_K \Hindman$.
\end{theorem}

\begin{proof}
By $\Gamma$ we will denote the set $\Gamma=\{(z_0,z_1)\in \omega^2:z_0>z_1\}$. In this proof we will view $\Ramsey$ as an ideal on $\Gamma$ consisting of those $A\subseteq\Gamma$ that do not contain any $B^2\cap\Gamma$, for infinite $B\subseteq\omega$.

By Lemma \ref{lem:VERY-SPARSE}, there is a very sparse $X\in[\omega]^\omega$. 
The ideal $\cH$ is homogeneous (see Proposition \ref{prop:homogeneous}), so it suffices to show that $\cR\not\leq_K\cH \restriction \FS(X)$. Fix any $f:\FS(X)\to \Gamma$ and assume to the contrary that it witnesses $\cR\leq_K\cH$. There are two possible cases.

\begin{case}
There are $k\in\omega$ and very sparse $D\in[\omega]^\omega$, $\FS(D)\subseteq \FS(X)$, such that for all $n>k$ and $x\in \FS(D)$ we have: 
$(f^{-1}[(\omega\times\{n\})\cap\Gamma]\cap\{y\in FS(D):\alpha_D(x)\subseteq \alpha_D(y)\})-x\in\cH\restriction \FS(X).$
\end{case}
In this case we recursively pick $\{x_n: n\in\omega\}\subseteq \FS(D)$ and $\{D_n : n\in\omega\cup\{-1\}\}\subseteq[\omega]^\omega$ such that $D_{-1}=D$ and for all $n\in\omega$ we have:
\begin{itemize}
    \item[(a)] $x_n\in \FS(D_{n-1})\setminus \big(\{x_i: i<n\}\cup \bigcup_{i<n}\bigcup_{j<n} \{y\in \FS(D_j): \alpha_{D_j}(y)\cap \alpha_{D_j}(x_i)\neq \emptyset\}\big)$ (here we put $\alpha_{D_j}(x_i)=\emptyset$ whenever $x_i\notin D_j$);
    \item[(b)] $D_n$ is very sparse;
    \item[(c)] $\FS(\{x_0,\ldots,x_n\})\subseteq \FS(D)$;
    \item[(d)] $\FS(D_n)\subseteq \FS(D_{n-1})\subseteq \FS(D)$;
    \item[(e)] $\left(f^{-1}[(\omega\times\{k+i\})\cap\Gamma]-x\right)\cap \FS(D_n)=\emptyset$ for every  $x\in \FS(\{x_0,\ldots,x_n\})$ and $1\leq i\leq n+1$;
    \item[(f)] $f^{-1}[(\omega\times\{k+i\})\cap\Gamma]\cap \FS(D_n)=\emptyset$ for all $1\leq i\leq n+1$.
\end{itemize}

The initial step of the construction is given by the requirement $D_{-1}=D$. Suppose now that  $x_i$ and $D_i$ for all $i<n$ are defined. 

Find $x_n\in \FS(D_{n-1})$ such that $\FS(\{x_0,\ldots,x_n\})\subseteq \FS(D)$ and $x_n\neq x_i$ for all $i<n$. This is possible since it suffices to pick any point from the set $$\FS(D_{n-1})\setminus \bigcup_{i<n}\bigcup_{j<n} \{y\in \FS(D_j): \alpha_{D_j}(y)\cap \alpha_{D_j}(x_i)\neq \emptyset\},$$ which is nonempty as $\FS(D_{n-1})\notin\cH \restriction \FS(X)$ and $$\bigcup_{i<n}\bigcup_{j<n} \{y\in \FS(D_j): \alpha_{D_j}(y)\cap \alpha_{D_j}(x_i)\neq \emptyset\}\in\cH\restriction \FS(X)$$
by Lemma \ref{lem:VERY-SPARSE2} and item (b) for all $j<n$ (here we put $\alpha_{D_j}(x_i)=\emptyset$ whenever $x_i\notin D_j$).

Enumerate $\FS(\{x_0,\ldots,x_n\})=\{c_0,c_1,\ldots,c_{2^{n+1}-2}\}$. We will define sets $E_t\in[\omega]^\omega$ for $-1\leq t\leq n$ such that $E_{-1}=D_{n-1}$ and for all $0\leq t\leq n$:
\begin{itemize}
    \item $\FS(E_{t})\subseteq \FS(E_{t-1})\subseteq \FS(D_{n-1})$,
    \item $\left(\bigcup_{1\leq i\leq n+1}f^{-1}[(\omega\times\{k+i\})\cap\Gamma]-c_l\right)\cap \FS(E_t)=\emptyset$ for every  $0\leq l\leq 2^{n+1}-2$.
\end{itemize}
Such construction is possible. Indeed, since we are on Case 1 and each $c_l\in \FS(\{x_0,\ldots,x_n\})\subseteq \FS(D)$, we know that:
\begin{equation*}
    \begin{split}
&\left(\bigcup_{1\leq i\leq n+1}f^{-1}[(\omega\times\{k+i\})\cap\Gamma]-c_l\right)
\cap\left(\{y\in \FS(D): \alpha_D(c_l)\subseteq \alpha_D(y)\}-c_l\right)
\\=&
\left(\bigcup_{1\leq i\leq n+1}f^{-1}[(\omega\times\{k+i\})\cap\Gamma]
\cap\{y\in \FS(D): \alpha_D(c_l)\subseteq \alpha_D(y)\}\right)-c_l
\in
\cH\restriction \FS(D).
\end{split}
\end{equation*}
On the other hand, we get:
\begin{equation*}
    \begin{split}
&
\FS(E_{t-1})\cap\left(\{y\in \FS(D):\alpha_D(c_l)\subseteq\alpha_D(y)\}-c_l\right)
\\
\supseteq &
\FS(E_{t-1})\cap\left(\FS(D)\setminus \{y\in \FS(D): \alpha_D(c_l)\cap \alpha_D(y)\neq\emptyset\}\right)
\\
\supseteq &
 \FS(E_{t-1})\setminus \{y\in \FS(D): \alpha_D(c_l)\cap\alpha_D(y)\neq\emptyset\}\notin\cH\restriction \FS(D),
\end{split}
\end{equation*}
as $\{y\in \FS(D):\alpha_D(c_l)\cap \alpha_D(y)\neq\emptyset\}\in\cH\restriction \FS(D)$ (by Lemma \ref{lem:VERY-SPARSE2}). Then 
\begin{equation*}
    \begin{split}
&\FS(E_{t-1})\setminus \left(\bigcup_{1\leq i\leq n+1}f^{-1}[
(\omega\times\{k+i\})\cap\Gamma]-c_l\right)
\\
\supseteq &
 \left(\FS(E_{t-1})\cap \left(\{y\in \FS(D): \alpha_D(c_l)\subseteq \alpha_D(y)\}-c_l\right)\right)
 \\ &
 \setminus   
 \left(\left(\bigcup_{1\leq i\leq n+1}f^{-1}[(\omega\times\{k+i\})\cap\Gamma]-c_l\right)\right.
\\ &
\cap 
 \left.\left(\{y\in \FS(D):\alpha_D(c_l)\subseteq \alpha_D(y)\}-c_l\right)\right)\notin\cH\restriction \FS(D).
 \end{split}
 \end{equation*}
 Thus, there is $E_t\in[\omega]^\omega$ as needed. 

Once all $E_t$ are defined, observe that
$$\bigcup_{1\leq i\leq n+1}(\omega\times\{k+i\})\cap\Gamma\in \cR.$$
Since we assumed that $f$ witnesses $\cR\leq_K\cH$, 
$$\FS(E_n)\setminus \bigcup_{1\leq i\leq n+1}f^{-1}[(\omega\times\{k+i\})\cap\Gamma]\notin \cH.$$ 
Hence, there is a very sparse $D_n\in [\omega]^{\omega}$ such that $$\FS(D_n)\subseteq \FS(E_n)\setminus \bigcup_{1\leq i\leq n+1}f^{-1}[(\omega\times\{k+i\})\cap\Gamma]$$
(by Lemma \ref{lem:VERY-SPARSE}).
Note that $\FS(D_n)\subseteq \FS(E_n)\subseteq \FS(D_{n-1})\subseteq \FS(D)$ and 
$$\bigcup_{1\leq i\leq n+1}\left(f^{-1}[(\omega\times\{k+i\})\cap\Gamma]-x\right)\cap \FS(D_n)=\emptyset$$ 
for all $x\in \FS(\{x_0,\ldots,x_n\})$.

This finishes the construction of $\{x_n: n\in\omega\}\subseteq \FS(D)$ and $\{D_n:n\in\omega\cup\{-1\}\}\subseteq[\omega]^\omega$.

Define $B=\FS(\{x_n:\ n\in\omega\})$. Obviously, $B\notin\cH\restriction \FS(X)$ as $\FS(\{x_0,\ldots,x_n\})\subseteq \FS(D)\subseteq \FS(X)$ for all $n\in\omega$. We will show that $f[B]\cap (\omega\times\{n\})\cap\Gamma$ is finite for all $n>k$. This will finish the proof in this case as $\bigcup_{n\leq k}(\omega\times\{n\})\cap\Gamma\in\cR$ and any set finite on each $(\omega\times\{n\})\cap\Gamma$ belongs to $\cR$.

Assume that $f(x)\in (\omega\times\{k+m+1\})\cap\Gamma$ for some $m\in\omega$ and $x=x_{n_0}+\ldots+x_{n_t}\in B$, where $n_0<\ldots<n_t$. If $n_0>m$, then $x\in \FS(\{x_n:\ n>m\})\subseteq \FS(D_m)$ which contradicts $f(x)\in (\omega\times\{k+m+1\})\cap\Gamma$ (by item (f)). If $n_0\leq m$ but $J=\{j\leq t: n_j>m\}\neq\emptyset$, then let $j=\min J$ and note that $x\in \sum_{i<j}x_{n_i}+\FS(D_m)$. As $\sum_{i<j}x_{n_i}\in \FS(\{x_0,\ldots,x_m\})$, item (e) gives us a contradiction with $f(x)\in (\omega\times\{k+m+1\})\cap\Gamma$. Hence, the only possibility is that $n_j\leq m$ for all $j\leq t$. Thus, $f[B]\cap(\omega\times\{k+m+1\})\cap\Gamma\subseteq f[\FS(\{x_0,\ldots,x_m\})]$, which is a finite set.

\begin{case}
For every $k\in\omega$ and very sparse $D\in[\omega]^\omega$, $\FS(D)\subseteq \FS(X)$, there are $n>k$ and $x\in \FS(D)$ such that: $$(f^{-1}[(\omega\times\{n\})\cap\Gamma]\cap\{y\in \FS(D): \alpha_D(x)\subseteq \alpha_D(y)\})-x\notin\cH\restriction \FS(X).$$
\end{case}
In this case we will pick $\{n_i: i\in \omega\}\subseteq\omega$, $\{j_i: i\in \omega\}\subseteq\{0,1\}$, $\{x_i: i\in \omega\}\subseteq \FS(D)$, $\{D_i: i\in \omega\cup\{-1\}\}\subseteq[\omega]^\omega$, $\{k_i: i\in \omega\}\subseteq\omega\cup\{-1\}$ and $\{F_i:i\in \omega\}\subseteq\fin$ such that $D_{-1}=X$ and for each $i\in\omega$:
\begin{itemize}
    \item[(a)] \begin{itemize}
    \item[(a1)] $n_i>n_{i-1}$ (here we put $n_{-1}=-1$);
    \item[(a2)] $n_{i}>\min\{a\in\omega: f[\FS(\{x_j: j<i\})]\subseteq \{0,1,\ldots,a\}^2\cap\Gamma\}$;
    \end{itemize} 
    \item[(b)] 
    \begin{itemize}
    \item[(b1)] $\FS(D_i)\subseteq \FS(D_{i-1})\subseteq \FS(X)$;
    \item[(b2)] $D_i$ is very sparse;
    \end{itemize}
    \item[(c)] if $j_i=0$, then:
    \begin{itemize}
        \item[(c1)] $k_i=-1$;
        \item[(c2)] $F_i=\emptyset$;
        \item[(c3)] $x_i\in \FS(D_{i-1})\cap f^{-1}[(\omega \times\{n_i\})\cap\Gamma]$;
        \item[(c4)] $x_i+\FS(D_i)\subseteq f^{-1}[(\omega\times\{n_i\})\cap\Gamma]$;
    \end{itemize}
    \item[(d)] if $j_i=1$, then:
    \begin{itemize}
        \item[(d1)] $k_i\in\{0\leq u<i: u\notin\bigcup_{q<i}F_q,j_u=0\}$; 
        \item[(d2)] $F_i=\{k_i,k_i+1,\ldots,i-1\}$;
        \item[(d3)] \begin{itemize}
            \item [(d3a)] $x_i\in f^{-1}[\{(n_i,n_{k_i})\}]$;
            \item [(d3b)] $x_i\in x_{k_i}+(\{0\}\cup\FS(\{x_r: k_i<r<i, r\notin\bigcup_{q<i}F_q\}))+\FS(D_{i-1})$;
        \end{itemize} 
        \item[(d4)] $x_i+\FS(D_i)\subseteq f^{-1}[\{(n_i,n_{k_i})\}]$;
    \end{itemize}
    \item[(e)] if $x=\sum_{b\leq a}x_{t_b}$ for some $0\leq t_0<\ldots<t_a<i$, $t_b\notin\bigcup_{q\leq i}F_q$ (so $x\in \FS(\{x_t: t<i,\ t\notin\bigcup_{q\leq i}F_q\})$), then: 
    \begin{itemize}
        \item[(e1)] $(x+x_i+\FS(D_i))\cap f^{-1}[\{(n_i,n_{t_0})\}]=\emptyset$;
        \item[(e2)] $(x+\FS(D_i))\cap f^{-1}[\{(n_i,n_{t_0})\}]=\emptyset$;
        \item[(e3)] $(x+x_i)\cap f^{-1}[\{(n_i,n_{t_0})\}]=\emptyset$;
        \end{itemize}
    \item[(f)] $\FS(D_i)\cap\{y\in \FS(D_{t}): \alpha_{D_{t}}(y)\cap \alpha_{D_{t}}(x_u)\neq\emptyset\}=\emptyset$ for all $-1\leq t<i$ and $0\leq u\leq i$ such that $x_u\in\FS(D_t)$;
    \item[(g)] \begin{itemize}
        \item[(g1)] $\FS(\{x_t: t\leq i, t\notin\bigcup_{q\leq i}F_q\})\subseteq \FS(X)$;
        \item[(g2)] $\sum_{b\leq a}x_{t_b}\in x_{t_{0}}+\FS(D_{t_0})$ for every $a>0$, $0\leq t_0<\ldots<t_a\leq i$, $t_b\notin\bigcup_{q\leq i}F_q$;
        \end{itemize}
\end{itemize}

At first step, since we are in Case 2, for $k=0$ and $D=X$ there are $n_0>k$ (note that (a) is satisfied) and $x'_0\in \FS(X)$ such that: 
$(f^{-1}[(\omega\times\{n_0\})\cap\Gamma]\cap\{y\in \FS(X): \alpha_X(x'_0)\subseteq \alpha_X(y)\})-x'_0\notin\cH\restriction \FS(X).$
Hence, there is $D'_0\in[\omega]^\omega$ such that: $x'_0+\FS(D'_0)\subseteq f^{-1}[(\omega\times\{n_0\})\cap\Gamma]\cap\{y\in \FS(X): \alpha_X(x'_0)\subseteq \alpha_X(y)\}\subseteq \FS(X).$
Put $j_0=0$, $k_0=-1$ and $F_0=\emptyset$ (note that (c1) and (c2) are satisfied). Moreover, define $x_0=x'_0+\min(D'_0)$ (note that (c3) and (g1) are satisfied, because $x_0\in x'_0+\FS(D'_0)\subseteq\FS(X)$ and $x'_0+\FS(D'_0)\subseteq f^{-1}[(\omega\times\{n_0\})\cap\Gamma]$) and using Lemma \ref{lem:VERY-SPARSE} find a very sparse $D_0\in[\omega]^\omega$ such that $$\FS(D_0)\subseteq \FS(D'_0\setminus\{\min(D'_0)\})\setminus \{y\in \FS(X): \alpha_X(y)\cap \alpha_X(x_0)\neq\emptyset\},$$ which is possible as $\{y\in \FS(X):\alpha_X(y)\cap \alpha_X(x_0)\neq\emptyset\}\in\cH\restriction \FS(X)$ by Lemma \ref{lem:VERY-SPARSE2} (note that (c4), (f) and (b) are satisfied, because $x_0+\FS(D_0)\subseteq x'_0+\FS(D'_0)\subseteq f^{-1}[(\omega\times\{n_0\})\cap\Gamma]$ and $\FS(D_0)\subseteq \FS(D'_0)\subseteq \{y\in \FS(X):\alpha_X(x'_0)\subseteq \alpha_X(y)\}-x'_0\subseteq \FS(X)$). In  conditions (e) and (g2) there is nothing to check. Thus, all the requirements are met.

At $i$th step, where $i>0$, since we are in Case 2, if $k=\max\{n_{i-1},\min\{a\in\omega: f[\FS(\{x_j: j<i\})]\subseteq \{0,1,\ldots,a\}^2\cap\Gamma\}\}$ 
and $D=D_{i-1}$, then there are $n_i>k$ (so (a) is satisfied) and $x'_i\in \FS(D_{i-1})$ such that $$\left(f^{-1}[(\omega\times\{n_i\})\cap\Gamma]\cap\{y\in \FS(D_{i-1}):\alpha_{D_{i-1}}(x_i')\subseteq \alpha_{D_{i-1}}(y)\}\right)-x'_i\notin\cH\restriction \FS(X).$$
Hence, there is $D'_i\in[\omega]^\omega$ such that: $x'_i+\FS(D'_i)\subseteq f^{-1}[(\omega\times\{n_i\})\cap\Gamma]\cap\{y\in \FS(D_{i-1}):\alpha_{D_{i-1}}(x_i')\subseteq \alpha_{D_{i-1}}(y)\}.$
In particular, $\FS(D'_i)\subseteq \{y\in \FS(D_{i-1}):\alpha_{D_{i-1}}(x_i')\subseteq \alpha_{D_{i-1}}(y)\}-x'_i= \{y\in \FS(D_{i-1}):\alpha_{D_{i-1}}(x_i')\cap\alpha_{D_{i-1}}(y)=\emptyset\}\subseteq\FS(D_{i-1})$. There are two possibilities. 

Assume first that there is $x=\sum_{b\leq a}x_{t_b}$ for some $t_0<\ldots<t_a<i$, $t_b\notin\bigcup_{q<i}F_q$ such that either $x+x'_i+\FS(\bar{D}_i)\subseteq f^{-1}[\{(n_i,n_{t_0})\}]$ or $x+\FS(\bar{D}_i)\subseteq f^{-1}[\{(n_i,n_{t_0})\}]$ for some $\bar{D}_i\in[\omega]^\omega$ such that $\FS(\bar{D}_i)\subseteq \FS(D'_i)$. Define $j_i=1$ and let $k_i$ be minimal such that there is (one or more) $x$ as above with $k_i=t_0$. 

Notice that $k_i\in\{0\leq u<i:\ u\notin\bigcup_{q<i}F_q\}$. We will show that $j_{k_i}=0$ (i.e., (d1) is satisfied). Suppose that $j_{k_i}=j_{t_0}=1$. Observe that $x'_i+\FS(\bar{D}_i)\subseteq \FS(D_{i-1})$ (by $\FS(\bar{D}_i)\cap\{y\in \FS(D_{i-1}):\alpha_{D_{i-1}}(y)\cap \alpha_{D_{i-1}}(x_i')\neq\emptyset\}=\emptyset$) and consequently $x'_i+\FS(\bar{D}_i)\subseteq \FS(D_{t_0})$ (by item (b1)). Then items (b1), (f) and (g2) give us: 
\begin{itemize}
    \item $x+x'_i+\FS(\bar{D}_i)\subseteq x_{t_0}+\FS(D_{t_0})$,
    \item $x+\FS(\bar{D}_i)\subseteq x_{t_0}+\FS(D_{t_0})$.
\end{itemize}
Then from (d4) we have:
\begin{itemize}
    \item $f[x+x'_i+\FS(\bar{D}_i)]\subseteq\{(n_{t_0},n_{k_{t_0}})\}$,
    \item $f[x+\FS(\bar{D}_i)]\subseteq\{(n_{t_0},n_{k_{t_0}})\}$.
\end{itemize}
This contradicts $x+x'_i+\FS(\bar{D}_i)\subseteq f^{-1}[\{(n_i,n_{t_0})\}]$ or $x+\FS(\bar{D}_i)\subseteq f^{-1}[\{(n_i,n_{t_0})\}]$, because $n_{t_0}<n_i$ (by $t_0<i$ and item (a1)).

Define $F_i=\{k_i,k_i+1,\ldots,i-1\}$ (so (d2) is satisfied) and $\bar{x}_i=x+x'_i$ (or $\bar{x}_i=x$ if $x+\FS(\bar{D}_i)\subseteq f^{-1}[\{(n_i,n_{t_0})\}]$). 

To define $D_i$ and $x_i$, note that by the choice of $k_i$, for each $y=\sum_{b\leq a}x_{t_b}$, $t_0<\ldots<t_a<i$, $t_b\notin\bigcup_{q\leq i}F_q$ (so in fact $t_a<k_i$) we know that $(y+\bar{x}_i+\FS(E))\not\subseteq f^{-1}[\{(n_i,n_{t_0})\}]$ and $(y+\FS(E))\not\subseteq f^{-1}[\{(n_i,n_{t_0})\}]$ for every $E\in[\omega]^\omega$ such that $\FS(E)\subseteq \FS(D'_i)$. In other words, $f^{-1}[\{(n_i,n_{t_0})\}]-(y+\bar{x}_i)\in\cH\restriction \FS(D'_i)$ and $f^{-1}[\{(n_i,n_{t_0})\}]-y\in\cH\restriction \FS(D'_i)$, for every such $y$. Thus, we can find $\tilde{D}_i\in[\omega]^\omega$ such that:
\begin{itemize}
    \item $\FS(\tilde{D}_i)\subseteq \FS(\bar{D}_i)$;
    \item $(y+\bar{x}_i+\FS(\tilde{D}_i))\cap f^{-1}[\{(n_i,n_{t_0})\}]=\emptyset$ and $(y+\FS(\tilde{D}_i))\cap f^{-1}[\{(n_i,n_{t_0})\}]=\emptyset$ for every $y=\sum_{b\leq a}x_{t_b}$, where $t_0<\ldots<t_a<i$, $t_b\notin\bigcup_{q\leq i}F_q$;
    \item $\FS(\tilde{D}_i)\cap\{y\in \FS(D_{i-1}):\alpha_{D_{i-1}}(y)\cap \alpha_{D_{i-1}}(x_i')\neq\emptyset\}=\emptyset$;
\end{itemize}
(the last item is trivial, as $\FS(\tilde{D}_i)\subseteq \FS(\bar{D}_i)\subseteq\FS(D'_i)$ and $\FS(D'_i)\subseteq\{y\in \FS(D_{i-1}):\alpha_{D_{i-1}}(x_i')\cap\alpha_{D_{i-1}}(y)=\emptyset\}$).

Define $x_i=\bar{x}_i+\min(\tilde{D}_i)$ and let $D_i\in[\omega]^\omega$ be very sparse such that $\FS(D_i)\subseteq \FS(\tilde{D}_i\setminus\{\min(\tilde{D}_i)\})$ and $D_i$ satisfies item (f). It is possible using Lemma \ref{lem:VERY-SPARSE}, as $\{y\in \FS(D_{t}):\alpha_{D_{t}}(y)\cap \alpha_{D_{t}}(x_u)\neq\emptyset\}\in\cH\restriction \FS(X)$ by Lemma \ref{lem:VERY-SPARSE2} and item (b2) for all $-1\leq t<i$. Then (b2) is satisfied. Observe that other conditions are met:
\begin{itemize}
    \item [(b1)] $\FS(D_i)\subseteq \FS(\tilde{D}_i)\subseteq \FS(\bar{D}_i)\subseteq \FS(D'_i)\subseteq \FS(D_{i-1})\subseteq \FS(X)$;
    \item [(d3b)] if $\bar{x}_i=x+x'_i$, then   
    $x_i=\bar{x}_i+\min(\tilde{D}_i)=x_{k_i}+(x-x_{k_i})+x'_i+\min(\tilde{D}_i)\in x_{k_i}+(\{0\}\cup\FS(\{x_r:\ k_i<r<i, r\notin\bigcup_{q<i}F_q\}))+\FS(D_{i-1})$ by the fact that $\FS(\tilde{D}_i)\cap\{y\in \FS(D_{i-1}):\alpha_{D_{i-1}}(y)\cap \alpha_{D_{i-1}}(x_i')\neq\emptyset\}=\emptyset$ (if $\bar{x}_i=x$ this is even easier to show);
    \item [(d3a)] if $\bar{x}_i=x+x'_i$, then $x_i\in x+x'_i+\FS(\bar{D}_i)\subseteq f^{-1}[\{(n_i,n_{k_i})\}]$ (if $\bar{x}_i=x$ this is also true);
    \item [(d4)] if $\bar{x}_i=x+x'_i$, then   $x_i+\FS(D_i)\subseteq x+x'_i+\FS(\bar{D}_i)\subseteq f^{-1}[\{(n_i,n_{k_i})\}]$ (if $\bar{x}_i=x$ this is also true);
    \item [(e)] for (e3), if $y=\sum_{b\leq a}x_{t_b}$, $t_0<\ldots<t_a<i$, $t_b\notin\bigcup_{q\leq i}F_q$, then note that $y+x_i=y+\bar{x}_i+\min(\tilde{D}_i)\in y+\bar{x}_i+\FS(\tilde{D}_i)$ and recall that $(y+\bar{x}_i+\FS(\tilde{D}_i))\cap f^{-1}[\{(n_i,n_{t_0})\}]=\emptyset$, thus $y+x_i\notin f^{-1}[\{(n_i,n_{t_0})\}]$ ((e1) and (e2) are similar);
    \item [(g1)] $\FS(\{x_t: t\leq i, t\notin\bigcup_{q\leq i}F_q\})\subseteq\FS(\{x_t: t<i, t\notin\bigcup_{q<i}F_q\})\cup(x_i+\FS(\{x_t: t<i, t\notin\bigcup_{q<i}F_q\})\subseteq \FS(X)$ by items (f) and (g1) applied to $i-1$ and item (d3b) applied to $i$;
    \item [(g2)] if $a>0$, $t_{0}<\ldots<t_{a}\leq i$, $t_b\notin\bigcup_{q\leq i}F_q$ then either $t_a<i$ and $\sum_{b\leq a}x_{t_b}\in x_{t_0}+\FS(D_{t_0})$ (by (g2) applied to $i-1$) or $t_a=i$ and $\sum_{b\leq a}x_{t_b}=\sum_{b<a}x_{t_b}+x_i\in x_{t_0}+(\{0\}\cup \FS(\{x_r: t_0<r<k_i, r\notin\bigcup_{q<i}F_q\}))+x_{k_i}+(\{0\}\cup \FS(\{x_r: k_i<r<i, r\notin\bigcup_{q<i}F_q\}))+\FS(D_{i-1})\subseteq x_{t_0}+\FS(D_{t_0})$ by items (b1), (d3b), (f) and (g2) for $i-1$.
\end{itemize}
Hence, all the requirements are met. This finishes the case of $j_i=1$.

Assume now that for all $x=\sum_{b\leq a}x_{t_b}$, $t_0<\ldots<t_a<i$, $t_b\notin\bigcup_{q<i}F_q$ we have $x+x'_i+\FS(E)\not\subseteq f^{-1}[\{(n_i,n_{t_0})\}]$ and $x+\FS(E)\not\subseteq f^{-1}[\{(n_i,n_{t_0})\}]$ for all $E\in[\omega]^\omega$ such that $\FS(E)\subseteq \FS(D'_i)$. Put $j_i=0$, $k_i=-1$ and $F_i=\emptyset$ (note that (c1) and (c2) are satisfied). 

Similarly as above (in the construction of $\tilde{D}_i$), we can find $\tilde{D}_i\in[\omega]^\omega$ such that:
\begin{itemize}
    \item $\FS(\tilde{D}_i)\subseteq \FS(D'_i)$;
    \item $(x+x'_i+\FS(\tilde{D}_i))\cap f^{-1}[\{(n_i,n_{t_0})\}]=\emptyset$ and $(x+\FS(\tilde{D}_i))\cap f^{-1}[\{(n_i,n_{t_0})\}]=\emptyset$ for every $x=\sum_{b\leq a}x_{t_b}$, $t_0<\ldots<t_a<i$, $t_b\notin\bigcup_{q\leq i}F_q$;
    \item $\FS(\tilde{D}_i)\cap\{y\in \FS(D_{i-1}):\alpha_{D_{i-1}}(y)\cap \alpha_{D_{i-1}}(x_i')\neq\emptyset\}=\emptyset$.
\end{itemize}

Define $x_i=x'_i+\min(\tilde{D}_i)$ and let $D_i\in[\omega]^\omega$ be very sparse such that $\FS(D_i)\subseteq \FS(\tilde{D}_i\setminus\{\min(\tilde{D}_i)\})$ and $D_i$ satisfies item (f) (which is possible by Lemmas \ref{lem:VERY-SPARSE} and \ref{lem:VERY-SPARSE2} and (b2) applied to all $-1\leq t<i$). Note that (b2) is satisfied. Observe that other conditions are met:
\begin{itemize}
    \item [(b1)] $\FS(D_i)\subseteq \FS(\tilde{D}_i)\subseteq \FS(D'_i)\subseteq \FS(D_{i-1})\subseteq \FS(X)$;
    \item [(c3)] $x_i\in \FS(D_{i-1})$ as $\FS(\tilde{D}_i)\cap\{y\in \FS(D_{i-1}):\alpha_{D_{i-1}}(y)\cap \alpha_{D_{i-1}}(x_i')\neq\emptyset\}=\emptyset$, $x_i\in x'_i+\FS(D'_i)\subseteq f^{-1}[(\omega \times\{n_i\})\cap\Gamma]$;
    \item [(c4)] $x_i+\FS(D_i)\subseteq x'_i+\FS(D'_i)\subseteq f^{-1}[(\omega \times\{n_i\})\cap\Gamma]$;
    \item [(e)] for (e3), if $x=\sum_{b\leq a}x_{t_b}$, $t_0<\ldots<t_a<i$, $t_b\notin\bigcup_{q\leq i}F_q$, then note that 
    $x+x_i=x+x'_i+\min(\tilde{D}_i)\in x+x'_i+\FS(\tilde{D}_i)$ and recall that $(x+x'_i+\FS(\tilde{D}_i))\cap f^{-1}[\{(n_i,n_{t_0})\}]=\emptyset$, thus $x+x_i\notin f^{-1}[\{(n_i,n_{t_0})\}]$
    ((e1) and (e2) are similar);
    \item [(g1)] by items (f) and (g1) applied to $i-1$ and item (c3) applied to $i$, $\FS(\{x_t: t\leq i,t\notin\bigcup_{q\leq i}F_q\})=\FS(\{x_t: t<i,t\notin\bigcup_{q<i}F_q\})\cup(x_i+\FS(\{x_t:\ t<i,t\notin\bigcup_{q< i}F_q\}))\subseteq \FS(X)$;
    \item [(g2)] if $a>0$, $t_{0}<\ldots<t_{a}\leq i$ and $t_b\notin\bigcup_{q\leq i}F_q$, then either $t_a<i$ and $\sum_{b\leq a}x_{t_b}\in x_{t_0}+\FS(D_{t_0})$ (by item (g2) applied to $i-1$) or $t_a=i$ and $\sum_{b\leq a}x_{t_b}\in x_{t_0}+\FS(D_{t_0})$ as $\sum_{b<a}x_{t_b}\in x_{t_0}+\FS(D_{t_0})$ and $x_i\in \FS(D_{i-1})\subseteq \FS(D_{t_0})\setminus \{y\in \FS(D_{t_0}):\ \exists_{j<i}\ \alpha_{D_{t_0}}(y)\cap \alpha_{D_{t_0}}(x_j)\neq\emptyset\}$ by item (c3) for $i$ and items (b1), (f) and (g2) for $i-1$.
    \end{itemize}
Hence, all the requirements are met. This finishes the case of $j_i=0$.

Note that $x_i\neq x_j$ for $i\neq j$ (it follows from items (a2), (c3) and (d3a)). Once the whole recursive construction is completed, define $A=\{x_i: i\notin\bigcup_{q\in\omega} F_q\}$. We need to show two facts:
\begin{itemize}
    \item[(i)] $A$ is infinite;
    \item[(ii)] $f[\FS(A)]\in\cR$.
\end{itemize}
Note that this will finish the proof as item (i) together with $\FS(A)\subseteq \FS(X)$ (by item (g1)) guarantee that $\FS(A)\notin\cH\restriction \FS(X)$.

(i): Since $x_i\neq x_j$ for $i\neq j$, we only need to show that there are infinitely many $t\in\omega$ such that $x_t\in A$. Assume to the contrary that there is $p\in\omega$ such that $x_t\notin A$ for all $t\geq p$. Without loss of generality we may assume that $p$ is minimal with that property. Since $x_p\notin A$, we have that $p\in\bigcup_{q\in\omega} F_q$, hence there is $q$ such that $p\in F_q$. By items (c2) and (d2), we know that $j_q=1$, $q>p$ and, by minimality of $p$, $F_q=\{p,p+1,\ldots q-1\}$. Again, as $x_q\notin A$ (because $q>p$), there should be $r$ such that $q\in F_r=\{k_r,k_r+1,\ldots,r-1\}$ (so $k_r\leq q<r$) and $k_r\geq p$ (by minimality of $p$). However, this is impossible as item (d1) gives us:
$$k_r\in\{u<r: u\notin\bigcup_{w<r} F_w, j_u=0\}\cap \{0,1,\ldots,q\}\subseteq$$
$$\subseteq\{u\leq q: u\notin F_q\}\cap\{u\leq q: j_u=0\}=$$
$$=(\{0,1,\ldots,p-1\}\cup\{q\})\cap\{u\leq q: j_u=0\}\subseteq \{0,1,\ldots,p-1\}$$
(here, if $p=0$ then $\{0,1,\ldots,p-1\}=\emptyset$).

(ii): We have:
\begin{equation*}
    \begin{split}
\FS(A)
=&
\bigcup_{i\in B}(\{x_i\}\cup(x_i+\FS(A\setminus\{0,1,\ldots,x_i\})))
\\ &\cup 
\bigcup_{i\in C}(\{x_i\}\cup(x_i+\FS(A\setminus\{0,1,\ldots,x_i\}))),
    \end{split}
\end{equation*}
where $B=\{i\in\omega: i\notin\bigcup_{q\in\omega}F_q, j_i=0\}$ and $C=\{i\in\omega: i\notin\bigcup_{q\in\omega}F_q, j_i=1\}$. 

At first we will show that $f[\bigcup_{i\in C}(\{x_i\}\cup(x_i+\FS(A\setminus\{0,1,\ldots,x_i\})))]\in\cR$. Note that $f(x_i)=(n_i,n_{k_i})$ (by item (d3a)) and $f[x_i+\FS(A\setminus\{0,1,\ldots,x_i\})]\subseteq f[x_i+\FS(D_i)]=\{(n_i,n_{k_i})\}$ for each $i\in C$ (by items (d4) and (g2)). Moreover, the sequence $(n_i)_{i\in \omega}$ is injective (by item (a1)). Hence, $f[\bigcup_{i\in C}(\{x_i\}\cup(x_i+\FS(A\setminus\{0,1,\ldots,x_i\})))]\in\cR$, as any set intersecting each $(\{n\}\times\omega)\cap\Gamma$ on at most one point belongs to $\cR$.

Now we will show that $f[\bigcup_{i\in B}(\{x_i\}\cup(x_i+\FS(A\setminus\{0,1,\ldots,x_i\})))]\in\cR$. By items (c3), (c4) and (g2), $f[\{x_i\}\cup(x_i+\FS(A\setminus\{0,1,\ldots,x_i\}))]\subseteq f[\{x_i\}\cup(x_i+\FS(D_i))]\subseteq (\omega\times\{n_i\})\cap\Gamma$ for all $i\in B$. Note that $\bigcup_{i\in B} f[\{x_i\}]\in \cR$ (from (a1), as each set intersecting each $(\omega\times\{n\})\cap\Gamma$ on at most one point belongs to $\cR$). Suppose that $Z^2\cap \Gamma\subseteq \bigcup_{i\in B} f[x_i+\FS(A\setminus\{0,1,\ldots,x_i\})]$ for some $Z\in[\omega]^\omega$. 

Firstly, we will show that $|Z\setminus \{n_i:i\in B\}|\leq 1$. Suppose that there are $z,w\in Z\setminus \{n_i:i\in B\}$ such that $z>w$. Then there is $i\in B$ such that $(z,w)\in f[x_i+\FS(A\setminus\{0,1,\ldots,x_i\})]$, hence $x_i+\FS(A\setminus\{0,1,\ldots,x_i\})\subseteq f^{-1}[\{(z,w)\}]$. But by (c4) and (g2) we have $x_i+\FS(A\setminus\{0,1,\ldots,x_i\})\subseteq x_i+\FS(D_i)\subseteq f^{-1}[(\omega\times\{n_i\})\cap\Gamma]$. So $(z,w)\in (\omega\times\{n_i\})\cap\Gamma$, i.e., $w=n_i$. A contradiction. 

By the previous paragraph, since $Z$ is infinite, there are $i,j\in B$ such that $j<i$ and $n_i,n_j\in Z$. We will show that $(n_i,n_j)\notin \bigcup_{k\in B} f[x_k+\FS(A\setminus\{0,1,\ldots,x_k\})]$. This will contradict $Z^2\cap \Gamma\subseteq \bigcup_{k\in B} f[x_k+\FS(A\setminus\{0,1,\ldots,x_k\})]$ and finish the proof.

Suppose that $(n_i,n_j)\in \bigcup_{k\in B} f[x_k+\FS(A\setminus\{0,1,\ldots,x_k\})]$. From (c4) and (g2), for every $k\neq j$, $k\in B$ we have $f[x_k+\FS(A\setminus\{0,1,\ldots,x_k\})]\subseteq f[x_k+\FS(D_k)] \subseteq(\omega\times\{n_k\})\cap\Gamma$, so $(n_i,n_j)\notin f[x_k+\FS(A\setminus\{0,1,\ldots,x_k\})]$. Hence, $(n_i,n_j)\in f[x_j+\FS(A\setminus\{0,1,\ldots,x_j\})]$. Let $y\in x_j+\FS(A\setminus\{0,1,\ldots,x_j\})$ be such that $f(y)=(n_i,n_j)$. Then $y=x_j+x_{s_0}+\ldots +x_{s_p}$ for some $j<s_0<\ldots <s_p$. We have five cases:
\begin{itemize}
    \item If $s_p<i$, then from item (a2) we have $f(y)\in [\{0,\ldots,n_i-1\}]^2$. A contradiction.
    \item If $s_p=i$, then $y=(x_j+\ldots+x_{s_{p-1}})+x_{i}$ and from item (e3) (applied to $x=(x_j+\ldots+x_{s_{p-1}})$) we get $f(y)\neq (n_i,n_j)$,   
    a contradiction.
    \item If there exists $k<p$ such that $s_k=i$, then $y=(x_j+\ldots+x_{s_{k-1}})+x_{i}+(x_{s_{k+1}}+\ldots+x_{s_p})\in (x_j+\ldots+x_{s_{k-1}})+x_{i}+\FS(D_i)$ by item (g2) and from item (e1) (applied to $x=(x_j+\ldots+x_{s_{k-1}})$) we get a contradiction.
    \item If there exists $k\leq p$ such that $s_{k-1}<i<s_k$, then $y=(x_j+\ldots+x_{s_{k-1}})+(x_{s_k}+\ldots+x_{s_p})\in (x_j+\ldots+x_{s_{k-1}})+\FS(D_i)$ by item (g2) and from item (e2) we get a contradiction.
    \item If $i<s_0$, then $y=x_j+(x_{s_0}+\ldots+x_{s_p})\in x_j+\FS(D_i)$ by item (g2) and from item (e2) we get a contradiction.
\end{itemize}

Thus, $f[\FS(A)]\in\cR$ and the proof is finished.
\end{proof}

%%%%%%%%%%%%%%%%%%%%%%%%%%%%%%%%%%%%%%%%%%%%%%%%%%
%%%%% REFERENCES
%%%%%%%%%%%%%%%%%%%%%%%%%%%%%%%%%%%%%%%%%%%%%%%%%%

\bibliographystyle{amsplain}
\bibliography{FKKK}

\end{document}